\documentclass[reqno,centertags, 12pt,draft]{amsart}
\usepackage{amsmath,amsthm,amscd,amssymb,enumerate}
\usepackage{latexsym}
\usepackage{graphicx}
\usepackage[mathscr]{eucal}
\usepackage{bm}
\newcommand{\bbC}{{\mathbb{C}}}
\newcommand{\bbD}{{\mathbb{D}}}

\newcommand{\bbN}{{\mathbb{N}}}

\newcommand{\bbR}{{\mathbb{R}}}

\newcommand{\bbZ}{{\mathbb{Z}}}

\newcommand{\fre}{{\frak{e}}}

\newcommand{\x}{{\mathbf{x}}}

\newcommand{\kk}{\mathbf{k}}

\newcommand{\La}{\Lambda}

\newcommand{\calM}{{\mathcal M}}

\newcommand{\calT}{{\mathcal T}}

\newcommand{\lb}{\label}
\newcommand{\f}{\frac}

\newcommand{\ti}{\tilde  }

\newcommand{\dist}{\text{\rm{dist}}}

\newcommand{\s}{\text{\rm{s}}}

\newcommand{\bi}{\bibitem}

\newcommand{\beq}{\begin{equation}}
\newcommand{\eeq}{\end{equation}}
\newcommand{\ba}{\begin{align}}
\newcommand{\ea}{\end{align}}
\newcommand{\eps}{\varepsilon}

\let\det=\undefined\DeclareMathOperator{\det}{det}

\newcounter{smalllist}
\newenvironment{SL}{\begin{list}{{\rm\roman{smalllist})}}{%
\setlength{\topsep}{0mm}\setlength{\parsep}{0mm}\setlength{\itemsep}{0mm}%
\setlength{\labelwidth}{2em}\setlength{\leftmargin}{2em}\usecounter{smalllist}%
}}{\end{list}}

\newcommand{\comm}[1]{}

\allowdisplaybreaks
\numberwithin{equation}{section}

\newtheorem{theorem}{Theorem}[section]

\newtheorem*{p2.1}{Proposition 2.1}

\newtheorem{lemma}[theorem]{Lemma}
\newtheorem{corollary}[theorem]{Corollary}
\theoremstyle{definition}
\newtheorem{example}[theorem]{Example}

\newtheorem*{remark}{Remark}

\newcommand{\abs}[1]{\lvert#1\rvert}

\begin{document}

\title[Finite Gap Jacobi Matrices, III]{Finite Gap Jacobi Matrices,\\III. Beyond the Szeg\H{o} Class}
\author[J.~S.~Christiansen, B.~Simon, and M.~Zinchenko]{Jacob S.~Christiansen$^{1}$, Barry Simon$^{2}$,\\and Maxim Zinchenko$^3$}

\thanks{$^1$ Department of Mathematical Sciences, University of Copenhagen,
Universitetsparken 5, DK-2100 Copenhagen, Denmark. E-mail: stordal@math.ku.dk.
Supported in part by a Steno Research Grant (09-064947) from the Danish Research Council for Nature and Universe}

\thanks{$^2$ Mathematics 253-37, California Institute of Technology, Pasadena, CA 91125.
E-mail: bsimon@caltech.edu. Supported in part by NSF grant DMS-0968856}

\thanks{$^3$ Department of Mathematics, University of Central Florida, Orlando, FL 32816.
E-mail: maxim@math.ucf.edu. Supported in part by NSF grant DMS-0965411}

\keywords{Szeg\H{o} asymptotics, orthogonal polynomials, almost periodic sequences, slowly decaying perturbations}
\subjclass[2010]{42C05, 39A11, 34L15}

\begin{abstract} Let $\fre\subset\bbR$ be a finite union of $\ell+1$ disjoint closed intervals and denote by $\omega_j$ the harmonic measure of the $j$ leftmost bands. The frequency module for $\fre$ is the set of all integral combinations of $\omega_1, \dots, \omega_\ell$. Let $\{\tilde{a}_n, \tilde{b}_n\}_{n=1}^\infty$ be a point in the isospectral torus for $\fre$ and $\tilde{p}_n$ its orthogonal polynomials. Let $\{a_n,b_n\}_{n=1}^\infty$ be a half-line Jacobi matrix with $a_n = \tilde{a}_n + \delta a_n$, $b_n = \tilde{b}_n + \delta b_n$. Suppose
\[
\sum_{n=1}^\infty 
\abs{\delta a_n}^2 + \abs{\delta b_n}^2 <\infty
\]
and $\sum_{n=1}^N e^{2\pi i\omega n} \delta a_n$, $\sum_{n=1}^N e^{2\pi i\omega n} \delta b_n$ have finite limits as $N\to\infty$ for all $\omega$ in the frequency module. If, in addition, these partial sums grow at most subexponentially with respect to $\omega$, then for $z\in\bbC\setminus\bbR$, $p_n(z)/\tilde{p}_n(z)$ has a limit as $n\to\infty$. Moreover, we show that there are non-Szeg\H{o} class $J$'s for which this holds.
\end{abstract}

\maketitle

\section{Introduction} \lb{s1}

Let $\fre$ in $\bbR$ be a compact set with $\ell +1$ intervals, that is,
\begin{equation} \lb{1.1}
\fre = \bigcup_{j=1}^{\ell+1} [\alpha_j,\beta_j], \qquad \alpha_1 < \beta_1 < \alpha_2 < \cdots < \beta_{\ell+1}
\end{equation}
$\ell$ counts the number of gaps, that is, bounded intervals in $\bbR\setminus\fre$. In this paper, we continue
our study of Jacobi matrices, $J$, and the asymptotics of the associated orthogonal polynomials on the real line
(OPRL) when the essential support of the spectral measure is $\fre$. In paper I \cite{CSZ1}, we discussed the isospectral
torus, $\calT_\fre$, associated to $\fre$, a family of two-sided almost periodic Jacobi matrices associated to $\fre$.
In paper II \cite{CSZ2}, we found equivalences among spectral conditions, one of which was a Szeg\H{o} condition,
and proved Szeg\H{o} asymptotics in this case. To explain our goal in this paper, let us recall the case $\ell=0$,
that is, a single band.

We thus consider Jacobi parameters $\{a_n,b_n\}_{n=1}^\infty$ and define
\begin{equation} \lb{1.2}
\delta b_n = b_n, \qquad \delta a_n = a_n-1
\end{equation}
(our free Jacobi matrix in this case has $\tilde{a}_n\equiv 1$, $\tilde{b}_n\equiv 0$). Recall that for this $\{\tilde{a}_n,
\tilde{b}_n\}_{n=1}^\infty$ and $n\geq 0$,
\begin{equation} \lb{1.3}
\widetilde{P}_n(x) = \f{z^{n+1} - z^{-n-1}}{z-z^{-1}}\, , \qquad z(x) = \f{x+\sqrt{x^2-4}\,}{2}
\end{equation}

The following is a result of Damanik--Simon \cite{Jost1}.

\begin{theorem} \lb{T1.1} Let $p_n(x)$ be the orthogonal polynomials associated to Jacobi parameters $\{a_n,b_n\}_{n=1}^\infty$
with $a_n\to 1$, $b_n\to 0$. Then the following are equivalent:
\begin{SL}
\item[{\rm{(i)}}] $\lim_{n\to\infty} p_n(x)/\tilde{p}_n(x)$ exists for all $x\in\bbC\setminus\bbR$ with convergence uniform on compact subsets.

\item[{\rm{(ii)}}] The Jacobi parameters obey
\begin{SL}
\item[{\rm{(a)}}] $\sum_{n=1}^\infty\, \abs{a_n-1}^2 + \abs{b_n}^2 <\infty$
\item[{\rm{(b)}}] $\sum_{n=1}^N (a_n-1)$ and $\sum_{n=1}^N b_n$ have limits in $(-\infty,\infty)$.
\end{SL}
\end{SL}
\end{theorem}

Damanik--Simon \cite{Jost1} also have examples where (i) and (ii) hold, but the Szeg\H{o} condition fails. They call
(i) $\Rightarrow$ (ii) ``the easy direction'' since it follows in a few lines from an analysis of the first few Taylor
coefficients at infinity for the $m$-functions. A key is the same miracle that makes Killip--Simon \cite{KS} work---that
for reasons we don't understand, a combination of the zeroth and second Taylor coefficients is nonnegative. (ii)
$\Rightarrow$ (i) is called ``the hard direction'' and \cite{Jost1} provides three proofs.

In this paper, our main goal is to prove a result akin to the hard direction of Theorem~\ref{T1.1} for perturbations
of elements of $\calT_\fre$ for some finite gap set, $\fre$. Given $\fre$, let $d\rho_\fre$ be the potential theoretic
equilibrium measure for $\fre$. Let $\omega_j =\rho_\fre ([\alpha_1,\beta_j])$ for $j=1, \dots, \ell$ and ${\bm\omega}=(\omega_1,\ldots,\omega_\ell)$. The frequency module, $\calM(\fre)$, for $\fre$ is the set of all numbers of the form
\begin{equation} \lb{1.4}
{\mathbf{k}}\cdot{\bm{\omega}} = \sum_{j=1}^{\ell} k_j \omega_j
\end{equation}
for ${\mathbf{k}}=(k_1,\ldots,k_\ell)\in\bbZ^\ell$. We'll only care about $e^{2\pi i({\mathbf{k}}\cdot\bm{\omega})n}$, $n\in\bbZ$, so only about $({\mathbf{k}}\cdot{\bm{\omega}})\bmod 1$. Thus, the
frequency module is essentially a subgroup of $\bbR/\bbZ$ with at most $\ell$ generators. In the periodic case, each
$\omega_j$ is of the form $m_j/p$, where $p$ is the period and there is no simpler common denominator. In that case,
in $\bbR/\bbZ$, $\calM(\fre)$ has $p$ elements $\{{m}/{p}\}_{m=0}^{p-1}$. If some $\omega_j$ is irrational,
$\calM(\fre)$ is infinite. Here is our main theorem in this paper:

\begin{theorem}\lb{T1.2} Let $\{\tilde{a}_n, \tilde{b}_n\}_{n=-\infty}^\infty$ be an element of the isospectral torus, $\calT_\fre$,
of a finite gap set, $\fre$. Let $\{a_n,b_n\}_{n=1}^\infty$ be another set of Jacobi parameters and $\delta a_n,\delta b_n$
given by
\begin{equation} \lb{1.5}
\delta a_n = a_n-\tilde{a}_n, \qquad \delta b_n = b_n-\tilde{b}_n
\end{equation}
Suppose that
\begin{SL}
\medskip
\item[{\rm{(a)}}]
\begin{equation} \lb{1.6}
\sum_{n=1}^\infty\, \abs{\delta a_n}^2 + \abs{\delta b_n}^2 <\infty
\end{equation}
\item[{\rm{(b)}}]
For any $\kk\in\bbZ^\ell$,
\begin{equation}
\lb{1.6a}
\sum_{n=1}^Ne^{2\pi i (\kk\cdot\bm\omega) n}\delta a_n
\quad\mbox{and}\quad
\sum_{n=1}^Ne^{2\pi i (\kk\cdot\bm\omega) n}\delta b_n
\end{equation}
have $($finite$)$ limits as $N\to\infty$.

\smallskip

\item[{\rm{(c)}}]
For every $\eps>0$,
\begin{equation} \lb{1.6b}
\sup_N\biggl\{\biggl\vert\sum_{n=1}^N e^{2\pi i({\bf k}\cdot \bm\omega)n} \delta a_n\biggr\vert +
\biggl\vert\sum_{n=1}^N e^{2\pi i({\bf k}\cdot \bm\omega)n} \delta b_n\biggr\vert \biggr\}
\leq C_\eps \exp(\eps\vert{\bf k}\vert)
\end{equation}

\smallskip

\end{SL}
Let $p_n$ {\rm{(}}resp.\ $\tilde{p}_n${\rm{)}} be the orthonormal
polynomials for $\{a_n,b_n\}_{n=1}^\infty$ {\rm{(}}resp.\ $\{\tilde{a}_n, \tilde{b}_n\}_{n=1}^\infty${\rm{)}}. Then for
any $x\in\bbC\setminus\bbR$,
\begin{equation} \lb{1.7}
\lim_{n\to\infty}\, \f{p_n(x)}{\tilde{p}_n(x)}
\end{equation}
exists and is finite and nonzero.
\end{theorem}

Thus, we have Szeg\H{o} asymptotics. We do not have the converse result (i.e., what \cite{Jost1} calls the easy direction) in part because one does not even have an analog of the Killip--Simon theorem \cite{KS} except in the case where the frequency module is $\{0,{1}/{p},{2}/{p}, \dots, {(p-1)}/{p}\}$ \cite{DKS}. We suspect the converse is true and think that even Szeg\H{o} asymptotics $\Rightarrow$ condition (b) would be interesting.

Another interesting object is the absolutely continuous spectrum. We believe that under the assumptions of Theorem~\ref{T1.2} it fills out the entire set $\fre$. In the single interval case, it was a consequence of the Killip--Simon theorem \cite{KS}. Using the ideas of Damanik--Simon \cite{Jost1}, one may attempt to control the Jost solution on the spectrum to show that the a.c.\ part of the spectral measure is supported on the finite gap set $\fre$. We will explore this idea elsewhere.

An example where the hypotheses of Theorem~\ref{T1.2} hold but the Szeg\H{o} condition fails is
\begin{equation} \lb{1.8}
\delta a_n = 0 \quad\text{and}\quad \delta b_n = \frac{1}{n^\alpha} \cos (2\pi\sqrt{n})
\end{equation}
with $\alpha\in(\f34,1)$. We refer to Section~\ref{s3} for details as well as an additional example.

\begin{remark}
In the periodic case the assumptions of Theorem~\ref{T1.2} become simpler. In particular, (c) is automatic for it follows from (b) since the frequency module has only finitely many elements. Moreover, in the period $p$ case (b) can be replaced by requiring
\begin{equation}
\sum_{n=0}^N \delta a_{np+j}
\quad\mbox{and}\quad
\sum_{n=0}^N \delta b_{np+j}
\end{equation}
to have $($finite$)$ limits as $N\to\infty$ for each $j=1,\dots,p$.
\end{remark}

\section{Szeg\H{o} Asymptotics} \lb{s2}

In this section we present a proof of Theorem \ref{T1.2} using a transfer matrix approach combined with a theorem of Coffman \cite{Co64}. The latter is a discrete version of an ODE result of Hartman--Wintner \cite{HW55} and reads:


\begin{theorem}[\cite{Co64}] \lb{T2.1}
Let $\La$ be a $d\times d$ diagonal invertible matrix with entries $\lambda_1,\dots,\lambda_d$ along the diagonal and let $\{A_n\}$ be a sequence of $d\times d$ matrices such that
\begin{align} \lb{2.1}
\sum_{n=1}^\infty \|A_n\|^2 < \infty
\end{align}
\begin{align} \lb{2.1a}
\La+A_n \;\text{ is invertible for all $n$}
\end{align}
Consider solutions $\vec{y}_n=(y_{n,1},\ldots,y_{n,d})\in\bbC^d$ of
\begin{align} \lb{2.2}
\vec{y}_{n+1}=(\La+A_n)\vec{y}_n
\end{align}
with some initial condition $\vec{y}_1$. Suppose $\lambda_j$ is a simple eigenvalue of $\Lambda$ with $|\lambda_k|\neq|\lambda_j|$ for all $k\neq j$ and let
\begin{align} \lb{2.3}
\gamma_j(n)=\prod_{k=n_0}^{n-1}(\lambda_j+(A_k)_{j,j})
\end{align}
where $n_0\geq 1$ is so large that $\gamma_j(n)\neq0$ for all $n>n_0$.
Then there exists an initial condition $\vec{y}_1$ so that
\begin{align} \lb{2.4}
\lim_{n\to\infty}\f{y_{n,j}}{\gamma_j(n)} = 1
\end{align}
while for $k\neq j$,
\begin{align} \lb{2.5}
\lim_{n\to\infty} \f{y_{n,k}}{\gamma_j(n)} = 0
\end{align}
\end{theorem}
Under the additional assumption of conditional summability of the perturbation, we get the following corollary of Coffman's result:
\begin{corollary} \lb{C2.2}
With $\La$ and $A_n$ as above, suppose that $|\lambda_1|> \dots > |\lambda_d|>0$ and
\begin{align} \lb{2.6}
\lim_{N\to\infty} \sum_{n=1}^N (A_n)_{j,j}
\end{align}
exists for every $j$. Then for any initial condition $\vec{y}_1\neq \bf{0}$, there exists $j$ such that the solution of \eqref{2.2} obeys
\begin{align} \lb{2.7}
\lim_{n\to\infty} \frac{{y}_{n,j}}{\lambda_j^n} = c_j \neq 0
\end{align}
while for $k\neq j$,
\begin{align} \lb{2.7a}
\lim_{n\to\infty} \frac{{y}_{n,k}}{\lambda_j^n} = 0
\end{align}
\end{corollary}

\begin{proof}
Recall that if $\sum_{k} |a_k|^2<\infty$ and $\sum_{k} a_k$ is conditionally convergent then
\begin{equation}
\sum_{k} \vert \log(1+a_k)-a_k \vert<\infty
\end{equation}
and hence $\prod_{k} (1+a_k)$ converges. Thus, it follows from \eqref{2.1}, \eqref{2.6}, and \eqref{2.3} that $\gamma_j(n)/\lambda_j^n$ has a finite nonzero limit, say $c_j$, as $n\to\infty$.

By Theorem \ref{T2.1}, there is for every $j$ a solution $\vec{y}_n(j)$ of \eqref{2.3} such that
\begin{equation}
\frac{\vec{y}_n(j)}{\gamma_j(n)}\to\vec{e}_j \,\mbox{ as }\, n\to\infty
\end{equation}
where $\{\vec{e}_1,\ldots,\vec{e}_d\}$ is the standard basis for $\bbC^d$.
Hence, $\vec{y}_n(j)/\lambda_j^n$ converges to $c_j\vec{e}_j$ as $n\to\infty$.

Note that the vectors $\vec{y}_n(1),\ldots,\vec{y}_n(d)$ are linearly independent for all $n$ and form a basis of $\bbC^d$.
Since $|\lambda_1| > |\lambda_2| > \dots > |\lambda_d|>0$, it follows that \eqref{2.7} holds for any solution $\vec{y}_n$
of \eqref{2.2}. For if $\vec{y}_n=\sum_{j=1}^d r_j\vec{y}_n(j)$, then
\begin{equation}
\frac{\vec{y}_n}{\lambda_k^n}=
\sum_{j=1}^d \frac{r_j\vec{y}_n(j)}{\lambda_j^n}\frac{\lambda_j^n}{\lambda_k^n}\to
r_k c_k \vec{e}_k \,\mbox{ as }\, n\to\infty
\end{equation}
where $k$ is the smallest value of $j$ for which $r_j\neq 0$.
\end{proof}

In what follows, $\x(z)$ will denote the universal covering map of $\bbD$ onto $\bbC\cup\{\infty\}\setminus\fre$ as introduced in \cite[Sect.~2]{CSZ1}. It is the unique meromorphic map which is locally one-to-one with
\begin{equation}
\x(z)=\frac{x_\infty}{z}+\mathcal{O}(1)
\end{equation}
near $z=0$ and $x_\infty>0$. Following \cite[Sect.~4]{CSZ1}, we let $B(z)$ be the Blaschke product with zeros at the poles of $\x$.
For every measure in the Szeg\H{o} class for $\fre$, one can define a Jost solution (cf. \cite[Sect.~9]{CSZ1}). While the parameters $\{a_n, b_n\}_{n=1}^\infty$ from Theorem \ref{T1.2} may not correspond to a measure in the Szeg\H{o} class, every point $\{\tilde{a}_n,\tilde{b}_n\}_{n=-\infty}^\infty$ in the isospectral torus does and its Jost solution, $\tilde{u}_n(z)$, satisfies the same three-term recurrence relation as $\tilde{p}_{n-1}(\x(z))$. We shall henceforth fix $z\in\bbD$ such that $\x(z)\in\bbC\setminus\bbR$.


Let us begin by writing the three-term recurrence relation for the orthonormal polynomials $p_n$ in matrix form
\begin{align} \lb{2.8}
\begin{pmatrix}
p_{n}(x)\\ a_n p_{n-1}(x)
\end{pmatrix}
=
\frac{1}{a_n}
\begin{pmatrix}
{x-b_n} & -1
\\
a_n^2 & 0
\end{pmatrix}
\begin{pmatrix}
p_{n-1}(x) \\
a_{n-1}p_{n-2}(x)
\end{pmatrix}
\end{align}
Similarly, 
\begin{equation} \lb{2.9}
\begin{aligned}
&\begin{pmatrix}
\tilde{p}_{n}(\x(z)) & \tilde{u}_{n+1}(z) \\
\tilde{a}_n \tilde{p}_{n-1}(\x(z)) & \tilde{a}_n \tilde{u}_n(z)
\end{pmatrix} \\
&\qquad=
\f{1}{\tilde{a}_n}
\begin{pmatrix}
{\x(z)-\tilde{b}_n} & -{1}
\\
\tilde{a}_n^2 & 0
\end{pmatrix}
\begin{pmatrix}
\tilde{p}_{n-1}(\x(z)) & \tilde{u}_n(z) \\
\tilde{a}_{n-1} \tilde{p}_{n-2}(\x(z)) & \tilde{a}_{n-1}\tilde{u}_{n-1}(z)
\end{pmatrix}
\end{aligned}
\end{equation}
To be dealing with bounded entries (for fixed $z$), we introduce
\begin{align} \lb{2.10}
R_n(z) =
\begin{pmatrix}
\tilde{p}_{n}(\x(z)) & \tilde{u}_{n+1}(z) \\
\tilde{a}_n \tilde{p}_{n-1}(\x(z)) & \tilde{a}_n \tilde{u}_n(z)
\end{pmatrix}
\begin{pmatrix}
B(z)^n & 0 \\
0 & B(z)^{-n}
\end{pmatrix}
\end{align}
Indeed, $B^{-n}\tilde{u}_n$ is almost periodic \cite[Thm.~9.2]{CSZ1} while $B^n\tilde{p}_{n-1}$ is almost periodic up to an exponentially small error \cite[Thm.~7.3]{CSZ2}.
By \eqref{2.9}, $\det R_n$ is $n$-independent. Hence,
\begin{equation}
\tilde{a}_n\bigl(\tilde{u}_n\tilde{p}_n-\tilde{u}_{n+1}\tilde{p}_{n-1}\bigr)=
\det R_n=\det R_0=\tilde{a}_0\tilde{u}_0\neq 0
\end{equation}
and so all $R_n$ are invertible. Changing the variables in \eqref{2.8} to $(\phi_n, \psi_n)$ defined by
\begin{align} \lb{2.11}
\begin{pmatrix}
\phi_n(z) \\
\psi_n(z)
\end{pmatrix}
=
R_n(z)^{-1}
\begin{pmatrix}
p_{n}(\x(z)) \\
a_n p_{n-1}(\x(z))
\end{pmatrix}
\end{align}
we get the recursion
\begin{align} \lb{2.12}
\begin{pmatrix}
\phi_n(z) \\
\psi_n(z)
\end{pmatrix}
=
\f{R_n(z)^{-1}}{a_n}
\begin{pmatrix}
{\x(z)-b_n} & -{1}
\\
a_n^2 & 0
\end{pmatrix}
R_{n-1}(z)
\begin{pmatrix}
\phi_{n-1}(z) \\
\psi_{n-1}(z)
\end{pmatrix}
\end{align}
It follows from \eqref{2.9}--\eqref{2.10} that
\begin{align} \lb{2.13}
R_n(z)^{-1} =
\begin{pmatrix}
B(z)^{-1} & 0 \\
0 & B(z)
\end{pmatrix}
\f{R_{n-1}(z)^{-1}}{\tilde{a}_n}
\begin{pmatrix}
0 & {1}
\\
-\tilde{a}_n^2 & {\x(z)-\tilde{b}_n}
\end{pmatrix}
\end{align}
and substitution of \eqref{2.13} into \eqref{2.12} leads to
\begin{align} \lb{2.14}
\begin{pmatrix}
\phi_n \\
\psi_n
\end{pmatrix}
=
\begin{pmatrix}
B^{-1} & 0 \\
0 & B
\end{pmatrix}
\left[
\begin{pmatrix}
1 & 0 \\
0 & 1
\end{pmatrix}
+
\bigl(R_{n-1}\bigr)^{-1}Q_nR_{n-1}\right]
\begin{pmatrix}
\phi_{n-1} \\
\psi_{n-1}
\end{pmatrix}
\end{align}
where
\begin{align}
\begin{split} \lb{2.14a}
Q_n(z) &= \f{1}{a_n \tilde{a}_n}
\begin{pmatrix}
0 & {1}
\\
-\tilde{a}_n^2 & {\x(z)-\tilde{b}_n}
\end{pmatrix}
\begin{pmatrix}
{\x(z)-b_n} & -{1}
\\
a_n^2 & 0
\end{pmatrix}
-
\begin{pmatrix}
1&0\\0&1
\end{pmatrix}
\\
& = \f{1}{a_n \tilde{a}_n}
\begin{pmatrix}
a_n\delta a_n & 0 \\[2mm]
(a_n+\tilde{a}_n)(\x(z)-\tilde{b}_n)\delta a_n + \tilde{a}_n^2 \delta b_n & -\tilde{a}_n \delta a_n
\end{pmatrix}
\end{split}
\end{align}
We wish to apply Corollary \ref{C2.2} with
\begin{equation} \lb{2.14b}
\Lambda=
\begin{pmatrix}
B^{-1} & 0 \\ 0 & B
\end{pmatrix}
\quad\mbox{and}\quad
A_n=
\begin{pmatrix}
B^{-1} & 0 \\ 0 & B
\end{pmatrix}
(R_{n-1})^{-1}Q_n R_{n-1}
\end{equation}
A straightforward computation shows that $B \det(R_{n-1})(A_n)_{1,1}$ can be written as
\begin{align}
\begin{split}
&\f{\tilde{p}_{n-1}}{\tilde{a}_n}
\Bigl(\tilde{a}_{n-1}\tilde{u}_{n-1}-\tilde{u}_n\bigl(\x(z)-\tilde{b}_n\bigr)\Bigr) \delta a_n \\
& \quad +
\f{\tilde{u}_n}{a_n}
\Bigl(\tilde{a}_{n-1}\tilde{p}_{n-2}-\tilde{p}_{n-1}\bigl(\x(z)-\tilde{b}_n\bigr)\Bigr)
\delta a_n -\frac{\tilde{a}_n}{a_n}\tilde{u}_n\tilde{p}_{n-1}\delta b_n \\
& \quad =
-\Bigl(\tilde{u}_{n+1}\tilde{p}_{n-1}+\f{\tilde{a}_n}{a_n}\tilde{u}_n\tilde{p}_n\Bigr) \delta a_n -
\frac{\tilde{a}_n}{a_n}\tilde{u}_n\tilde{p}_{n-1}\delta b_n
\end{split}
\end{align}
so that
\begin{equation} \lb{2.15}
\begin{aligned}
B (A_n)_{1,1}&=
\biggl( \frac{1/\tilde{a}_n}{1-\tilde{u}_{n}\tilde{p}_n/(\tilde{u}_{n+1}\tilde{p}_{n-1})}-
 \frac{1/{a}_n}{1-\tilde{u}_{n+1}\tilde{p}_{n-1}/(\tilde{u}_{n}\tilde{p}_{n})}\biggr) \delta a_n \\
&\qquad +\frac{1/a_n}{\tilde{u}_{n+1}/\tilde{u}_n-\tilde{p}_n/\tilde{p}_{n-1}}\delta b_n
\end{aligned}
\end{equation}
Moreover,
\begin{equation} \lb{2.16}
B^{-1}(A_n)_{2,2}=\frac{(\delta a_n)^2}{a_n\tilde{a}_n}-B(A_n)_{1,1}
\end{equation}
To prove that \eqref{2.6} holds, we need the following lemma.
\begin{lemma} \lb{L2.3}
For any real analytic almost periodic sequence $\{f_n\}_{n=1}^\infty$ with frequency module contained in $\calM(\fre)$,
\begin{equation} \lb{2.17}
\sum_{n=1}^N f_n\, \delta a_n \quad\text{and}\quad \sum_{n=1}^N f_n\, \delta b_n
\end{equation}
have limits as $N\to\infty$.
\end{lemma}

\begin{proof}
For simplicity, we only consider $\delta a_n$. By assumption,
\begin{equation}
f_n=\sum_{\mathbf{k}}c_{\kk}e^{2\pi i(\kk\cdot\bm\omega)n}
\end{equation}
where the Fourier coefficients $c_\kk$ obey
\begin{equation}
\label{exp}
\vert c_\kk\vert\leq C e^{-D\vert\kk\vert}
\end{equation}
for some $C, D>0$. Given $\eps>0$, we will show that
\begin{equation}
\label{NM}
\left\vert \sum_{n=1}^N f_n \,\delta a_n -
\sum_{n=1}^M f_n \,\delta a_n \right\vert<\eps
\end{equation}
for $N, M$ sufficiently large. Assume $N>M$ and start by taking $K$ so large that
\begin{equation}
\label{k big}
\sum_{\vert\kk\vert>K} \vert c_\kk \vert
\left\vert\sum_{n=M+1}^N e^{2\pi i(\kk\cdot\bm\omega) n}\delta a_n \right\vert<\eps/2
\end{equation}
for all $N, M$. We can do so using \eqref{1.6b} and \eqref{exp}. Next, take $N, M$ so large that
\begin{equation}
\label{large}
\vert c_\kk \vert
\left\vert\sum_{n=M+1}^N e^{2\pi i(\kk\cdot\bm\omega) n}\delta a_n \right\vert
<\frac{\eps/2}{\#\bigl\{\kk : \vert\kk\vert\leq K\bigr\}}
\end{equation}
for all $\kk$ with $\vert\kk\vert\leq K$. Since only finitely many $\kk$'s occur here, this can be done by \eqref{1.6a}.
Combining \eqref{k big} and \eqref{large} leads to \eqref{NM} and the result follows.
\end{proof}

By \cite[Thm.~7.3]{CSZ2}, the sequence of polynomials $\ti p_{n-1}$ can be written as a linear combination,
\begin{align} \lb{2.21}
\ti p_{n-1} = r_1  B^{-n}\ti v_n + r_2 \ti e_n
\end{align}
with $r_j\neq0$, $j=1,2$, where $\ti e_n$ is exponentially decaying and $\ti v_n$ is almost periodic. In fact, $\ti{a}_{n+1}\ti v_n$ is given by the Jost function (cf.\ \cite[Sect.~8]{CSZ1}) sampled along an equally spaced orbit on the isospectral torus $\calT_\fre$, hence $|\ti v_n| \ge c >0$. Thus, \eqref{2.15} can be rewritten as
\begin{equation} \lb{2.22}
\begin{aligned}
B (A_n)_{1,1}&=
\biggl( \frac{1/\tilde{a}_n}{1-\tilde{u}_{n}\tilde{v}_{n+1}/(B\tilde{u}_{n+1}\tilde{v}_{n})} -
 \frac{1/{a}_n}{1-B\tilde{u}_{n+1}\tilde{v}_{n}/(\tilde{u}_{n}\tilde{v}_{n+1})}\biggr) \delta a_n \\
&\qquad + \frac{1/a_n}{\tilde{u}_{n+1}/\tilde{u}_n-\tilde{v}_{n+1}/(B\tilde{v}_{n})}\delta b_n + e_n
\end{aligned}
\end{equation}
where $e_n$ is again an exponentially decaying sequence. Using the identity $1/a_n = 1/\ti a_n - \delta a_n/(a_n\ti a_n)$, we can rewrite \eqref{2.22} once more as
\begin{equation} \lb{2.23}
\begin{aligned}
B (A_n)_{1,1}&=
\biggl( \frac{1/\tilde{a}_n}{1-\tilde{u}_{n}\tilde{v}_{n+1}/(B\tilde{u}_{n+1}\tilde{v}_{n})} -
 \frac{1/\tilde{a}_n}{1-B\tilde{u}_{n+1}\tilde{v}_{n}/(\tilde{u}_{n}\tilde{v}_{n+1})}\biggr) \delta a_n \\
&\qquad + \frac{1/\tilde{a}_n}{\tilde{u}_{n+1}/\tilde{u}_n-\tilde{v}_{n+1}/(B\tilde{v}_{n})}\delta b_n + e_n + s_n
\end{aligned}
\end{equation}
Here $s_n$ involves two terms, each given by a product of either $(\delta a_n)^2$ or $\delta a_n\delta b_n$ and a bounded factor. Hence the sequence $s_n$ is absolutely summable by \eqref{1.6}.

Now the definitions of the Jost function in \cite[Sect.~8]{CSZ1} and the Jost solution in \cite[Sect.~9]{CSZ1} combined with \cite[Cor.~6.4]{CSZ1} show that $1/\ti a_n$, $\ti u_{n+1}/\ti u_n$, and $\ti v_{n+1}/\ti v_{n}$ are real analytic quasiperiodic sequences. So \eqref{2.16}, \eqref{2.23}, absolute summability of $e_n$ and $s_n$, and Lemma~\ref{L2.3} imply that
\begin{equation}
\sum_{n=1}^N (A_n)_{1,1} \quad\text{and}\quad \sum_{n=1}^N (A_n)_{2,2}
\end{equation}
have limits as $N\to\infty$. Moreover, \eqref{2.14a}--\eqref{2.14b} combined with  \eqref{1.6} imply that
\begin{equation}
\sum_{n=1}^\infty \big\|A_n\|^2 < \infty
\end{equation}
Thus, the recursion \eqref{2.14} satisfies the conditions of Corollary \ref{C2.2} and hence we either have
\begin{align} \lb{2.27}
\lim_{n\to\infty}
B^n \begin{pmatrix}\phi_n\\\psi_n\end{pmatrix}
=c_1\begin{pmatrix}1\\0\end{pmatrix}
\end{align}
or
\begin{align}
\label{2.27a}
\lim_{n\to\infty}
B^{-n} \begin{pmatrix}\phi_n\\\psi_n\end{pmatrix}
=c_2\begin{pmatrix}0\\1\end{pmatrix}
\end{align}
for some nonzero constants $c_1, c_2$.
The latter, however, is impossible. For since $B^{n}\tilde{p}_{n-1}$ and $B^{-n}\tilde{u}_n$ are bounded, \eqref{2.27a} would, by \eqref{2.10} and \eqref{2.11}, imply that
\begin{align} \lb{2.28}
p_n(\x(z)) = \phi_n(z)B(z)^n \tilde{p}_n(\x(z)) + \psi_n(z) B(z)^{-n} \tilde{u}_{n+1}(z)
\end{align}
decays exponentially as $n\to\infty$. Hence, the associated Jacobi matrix has an eigenvalue at $\x(z)$ which by assumption is a point in $\bbC\setminus\bbR$. This cannot be true and the first alternative \eqref{2.27} must therefore hold.
Rewriting \eqref{2.28} as
\begin{align}
\f{p_{n}(\x(z))}{\tilde{p}_n(\x(z))} =
B(z)^n\phi_n(z) + \f{\tilde{u}_{n+1}(z)B(z)^{-n}}{\tilde{p}_{n}(\x(z))B(z)^{n}}B(z)^{n}\psi_n(z)
\end{align}
then leads to the conclusion of Theorem \ref{T1.2}.

\section{Examples} \lb{s3}

In this section we supplement Theorem~\ref{T1.2} with examples for which \eqref{1.6}--\eqref{1.6b} hold. In particular, we give an example showing that the Szeg\H{o} condition is not necessary for Szeg\H{o} asymptotics to hold. To put our results in the right perspective, recall that, by \cite[Thm.~7.4]{CSZ2}, Szeg\H{o} asymptotics \eqref{1.7} holds when the measure of orthogonality $d\mu(x)=w(x)dx+d\mu_\s(x)$ has essential support $\fre$ and belongs to the Szeg\H{o} class for $\fre$, that is,
\begin{align} \lb{3.1}
&\int_\fre \f{\log(w(x))}{\dist(x,\bbR\setminus\fre)^{1/2}}\, dx >-\infty &\text{(Szeg\H{o} condition)}
\intertext{and}
\lb{3.2}
&\sum_k \dist (x_k,\fre)^{1/2} <\infty  &\text{(Blaschke condition)}
\end{align}
where $\{x_k\}$ denote the mass points of $d\mu$ in $\bbR\setminus\fre$.

In continuation of \cite{FSW} and \cite{CSZ2}, it was shown in \cite[Thm.~1.3]{FS11} that the generalized Nevai conjecture holds, that is, the measure of orthogonality is in the Szeg\H{o} class if the sequence of recurrence coefficients $\{a_n,b_n\}_{n=1}^\infty$ is an $\ell^1$-perturbation of some element $\{\tilde{a}_n, \tilde{b}_n\}_{n=-\infty}^\infty$ in $\calT_\fre$. Thus,
\begin{equation} \lb{3.3}
\sum_{n=1}^\infty\, \abs{a_n - \ti a_n} + \abs{b_n - \ti b_n} <\infty
\end{equation}
is a sufficient condition for Szeg\H{o} asymptotics. Our main result, Theorem~\ref{1.2}, shows that \eqref{3.3} can be relaxed and the weaker set of conditions \eqref{1.6}--\eqref{1.6b} is also sufficient for Szeg\H{o} asymptotics to hold.

In the following example we use Theorem~\ref{T1.2} to show that \eqref{1.7} may hold for a non $\ell^1$-perturbation of an element in $\calT_\fre$; hence, the $\ell^1$-condition \eqref{3.3} is not necessary for Szeg\H{o} asymptotics.

\begin{example}
Let $\{\tilde{a}_n, \tilde{b}_n\}_{n=-\infty}^\infty$ be an arbitrary element of the isospectral torus $\calT_\fre$. Pick $\alpha\in(\f12,1)$, $\omega\in(0,1)$, and let
\begin{align} \lb{3.4a}
a_n=\ti a_n+\delta a_n, \quad b_n=\ti b_n + \delta b_n
\end{align}
with
\begin{align} \lb{3.5a}
\delta a_n = 0, \quad \delta b_n=\f{1}{n^\alpha}\cos(2\pi \omega n)
\end{align}
In addition, assume that the almost periods $\omega_j =\rho_\fre ([\alpha_1,\beta_j])$ obey a Diophantine condition relative to the frequency $\omega$, that is, there exist a constant $C>0$ and an integer $s$ such that for all $\bm k = (k_1, \dots, k_\ell)\in\bbZ^\ell$,
\begin{equation} \lb{3.6a}
\bigg\{ \pm\omega + \sum_{j=1}^\ell k_j \omega_j \,\bigg\} \geq \f{C}{(1+\abs{\bm k})^{s}}
\end{equation}
where $\{x\}=x\bmod1$ denotes the fractional part of $x$. It is known that the Diophantine condition is satisfied for Lebesgue a.e.\ $\bm\omega$. Hence, by a theorem of Totik \cite{To01}, Lebesgue a.e.\ $\{\alpha_j,\beta_j\}_{j=1}^{\ell+1}$ lead to the condition \eqref{3.6a}.

We start by noting that \eqref{1.6} is trivially satisfied since $\alpha>\f12$. On the other hand, it is easy to see that \eqref{3.3} fails to hold since $\alpha<1$. Next, we verify \eqref{1.6a}--\eqref{1.6b}. It follows from Theorem~2.6 in \cite[Chap.~I]{Zy88} that
\begin{align} \lb{3.7a}
\sum_{n=1}^\infty \f{e^{2\pi i x n}}{n^\alpha}
\end{align}
converges uniformly for $x$ in $[\eps,1-\eps]$ for every $\eps>0$. Moreover, by (2.26)--(2.27) in \cite[Chap.~V]{Zy88} there exists a constant $D>0$ so that for all $x\in[\eps,1-\eps]$,
\begin{align} \lb{3.7b}
\sup_N \bigg| \sum_{n=1}^N \f{e^{2\pi i x n}}{n^\alpha} \,\bigg| < \f{D}{\eps}
\end{align}
So taking $x=\pm\omega+\bm k\cdot\bm\omega$ in \eqref{3.7a} yields convergence of the series
\begin{align} \lb{3.8a}
\sum_{n=1}^\infty \, \f{\cos(2\pi \omega n)}{n^\alpha} \, e^{2\pi i (\bm k\cdot\bm\omega) n}
\end{align}
for all $\bm k\in\bbZ^\ell$ since by \eqref{3.6a}, $x\neq0\bmod 1$. Hence, \eqref{1.6a} holds. Similarly, taking $x=\pm\omega+\bm k\cdot\bm\omega$ in \eqref{3.7b} and utilizing \eqref{3.6a} gives
\begin{align}
\sup_{N\in\bbN} \left|\, \sum_{n=1}^N \, \f{\cos(2\pi \omega n)}{n^\alpha} \, e^{2\pi i (\bm k\cdot\bm\omega) n} \,\right| \le \f{D}{C}(1+|\bm k|)^s
\end{align}
Hence, \eqref{1.6b} is also satisfied. Thus we have Szeg\H{o} asymptotics by Theorem~\ref{T1.2}.
\end{example}

Our second example illustrates 
that neither the $\ell^1$-condition \eqref{3.3} nor belonging to the Szeg\H{o} class is necessary for Szeg\H{o} asymptotics to hold.


\begin{example}
Let $\{\tilde{a}_n, \tilde{b}_n\}_{n=-\infty}^\infty$ be an arbitrary element of the isospectral torus $\calT_\fre$. Pick $\alpha\in(\f34,1)$ and let $a_n$, $b_n$ be given as in \eqref{3.4a} but now with
\begin{align} \lb{3.5}
\delta a_n = 0, \quad \delta b_n=\f{1}{n^\alpha}\cos(2\pi \sqrt{n})
\end{align}
Then \eqref{1.6} is trivially satisfied since $\alpha>\f12$ and \eqref{3.3} fails to hold since $\alpha<1$. To show that \eqref{1.6a}--\eqref{1.6b} are fulfilled, we note that Theorem~5.2(i) in \cite[Chap.~V]{Zy88} with $2\pi$ adjusted in both exponents implies that
\begin{align}
\sum_{n=1}^\infty \f{e^{2\pi i \sqrt{n}} }{n^\alpha} \, e^{2\pi i n x}
\end{align}
converges uniformly in $x$ to a continuous function on $[-\f12,\f12]$. Combining real and imaginary parts of this series at $x=\pm\omega$ yields convergence of the series
\begin{align} \lb{3.6}
\sum_{n=1}^\infty \, \f{\cos(2\pi\sqrt{n})}{n^\alpha} \, e^{2\pi in\omega}
\end{align}
for all $\omega\in[0,1]$. Moreover, the partial sums are uniformly bounded
\begin{align} \lb{3.7}
\sup_{\omega\in[0,1]} \, \sup_{N\in\bbN}
\left|\, \sum_{n=1}^N \, \f{\cos(2\pi\sqrt{n})}{n^\alpha} \, e^{2\pi in\omega} \,\right| \le C
\end{align}
Since \eqref{3.6}--\eqref{3.7} imply \eqref{1.6a}--\eqref{1.6b}, we have Szeg\H{o} asymptotics by Theorem~\ref{T1.2}.

Next, we show that the eigenvalues $\{x_k\}$ in $\bbR\setminus\fre$ of the Jacobi matrix $J$ with parameters $\{a_n, b_n\}_{n=1}^\infty$ satisfy
\begin{equation} \lb{3.8}
\sum_k \dist (x_k,\fre)^q = \infty
\end{equation}
for any $q<\f{1}{2\alpha}$, in particular, for $q=\f12$. Hence, the Blaschke condition \eqref{3.2} fails and, by \cite[Thm.~4.1]{CSZ2}, also the Szeg\H{o} condition \eqref{3.1} fails to hold in the present example.

Indeed, let
\begin{equation}
B_m=\bigl[(8m)^2-m,(8m)^2+m\bigr], \quad m\ge1
\end{equation}
denote disjoint subsets of $\bbN$. Recall that, by
\cite[Thm.~9.6]{CSZ1}, the unperturbed difference equation
\begin{equation} \lb{3.9}
\ti a_n w_{n+1} + (\ti b_n-\beta_{\ell+1}) w_n + \ti a_{n-1} w_{n-1} =0
\end{equation}
has a positive bounded solution $u_n$ so that
\begin{align} \lb{3.9a}
c_1 \le u_n \le c_2
\end{align}
uniformly in $n$ with some fixed positive constants $c_1$, $c_2$.
In addition, by \cite[Thm.~9.7]{CSZ1}, \eqref{3.9} also has an increasing solution $v_n$ which, by \cite[Eq.~(9.31)]{CSZ1}, is of the form
\begin{align} \lb{3.9b}
v_n = \kappa n u_n + s_n
\end{align}
where $\kappa$ is a fixed positive constant and $\{s_n\}$ a bounded
real-valued sequence. Let $c_3$ be an upper bound for $|s_n|$ and
$\mu_{\mp}=(8m)^2\mp m$ be the left (resp.\ right) endpoint of $B_m$. Then,
by \eqref{3.9a}--\eqref{3.9b}, the linear combinations
\begin{align} \lb{3.9c}
\pm (v_n  - \kappa\mu_{\mp} u_n) + (c_3/c_1)u_n
\end{align}
are positive on $B_m$ and of magnitude at least $m \kappa c_1$ at the
center of $B_m$ and at most $c_3(1+c_2/c_1)$ at the left (resp.\
right) endpoint of $B_m$. Rescaling the expression in \eqref{3.9c}, we
obtain solutions $v^{(m)}_{\pm}$ of \eqref{3.9} with the following
properties:


\begin{itemize}
\item
$v^{(m)}_{\pm}$ are positive on $B_m$ and equal to $1$ at the center of $B_m$,
\smallskip
\item
$v^{(m)}_{+}$ is of order $\mathcal{O}(1/m)$ at the left endpoint of $B_m$, and
\smallskip
\item
$v^{(m)}_{-}$ is of order $\mathcal{O}(1/m)$ at the right endpoint of $B_m$.
\end{itemize}
Now define $\phi^{(m)}$ by
\begin{equation*}
\bigl(\phi^{(m)}\bigr)_n=\begin{cases}
\bigl(v^{(m)}_+\bigr)_n & \mbox{if } n\in\bigl[(8m)^2-m,(8m)^2\bigr] \\
\bigl(v^{(m)}_-\bigr)_n & \mbox{if } n\in\bigl[(8m)^2,(8m)^2+m\bigr] \\
\quad\, 0 & \mbox{if } n\not\in B_m
\end{cases}
\end{equation*}
Then there is a constant $C>0$ (independent of $m$) such that
\begin{align}
\|\phi^{(m)}\|^2 \ge C m
\end{align}
As $\bigl\vert\sqrt{(8m)^2\pm m}-8m\bigr\vert<\f18$, we have $\delta b_n > C'm^{-2\alpha}$ for $n\in B_m$ so that
\begin{align}
\bigl\langle \phi^{(m)}, (J-\tilde{J})\phi^{(m)}\bigr\rangle \ge \f{C' \|\phi^{(m)}\|^2}{m^{2\alpha}}
\end{align}
Moreover, since $\phi^{(m)}$ satisfies \eqref{3.9} for all $n$ except at the center and near the endpoints of $B_m$,
\begin{align}
\bigl\vert\bigl\langle \phi^{(m)}, (\ti J - \beta_{\ell+1}) \phi^{(m)}\bigr\rangle\bigr\vert \le \f{C''}{m}
\end{align}
Combining the above inequalities gives
\begin{align}
\f{\bigl\langle \phi^{(m)}, (J - \beta_{\ell+1}) \phi^{(m)}\bigr\rangle}{\|\phi^{(m)}\|^2} \ge \f{C'}{m^{2\alpha}} - \f{C''/C}{m^2}
\end{align}
Hence, for large $m$ and some small constant $D>0$, we have
\begin{align}
\dist(x_m,\fre) \ge \f{D}{m^{2\alpha}}
\end{align}
Thus, \eqref{3.8} holds for all $q<\f{1}{2\alpha}$.
\end{example}

\begin{remark}
With the additional Diophantine assumption \eqref{3.6a} and somewhat more delicate estimates as in \cite[Sect.~9]{Jost1}, one can show that there are examples where \eqref{3.8} holds with $q$ arbitrarily close to $\f32$, $q<\f32$.
\end{remark}


\medskip

\noindent {\bf Acknowledgements.}
J.S.C. and M.Z. gratefully acknowledge the kind invitation and hospitality of the Mathematics Department of Caltech where this work was completed.

\medskip


\end{document}